\documentclass[symmetry,article,submit,oneauthor,10pt,a4paper]{mdpiarXiv}

\firstpage{1} %%MDPI internal note: new layout%%
\articlenumber{x}
\doinum{10.3390/------}
\pubvolume{xx}
\pubyear{2016}
\copyrightyear{2016}
\externaleditor{Academic Editor: Charles F. Dunkl}
\history{Received: date; Accepted: date; Published: date}
\setcitestyle{authoryear,open={(},close={)}}
\usepackage{bm}

 \theoremstyle{mdpi}
 \newcounter{thm}
 \newcounter{ex}
 \newcounter{re}
 \newtheorem{Theorem}[thm]{Theorem}

 \theoremstyle{mdpidefinition}

\Title{Multivariate Krawtchouk polynomials and composition birth and death processes}

\Author{Robert Griffiths}
\address{%
\quad Affiliation; Department of Statistics, University of Oxford, UK, email; griff@stats.ox.ac.uk}
\abstract{
This paper defines the multivariate Krawtchouk polynomials, orthogonal on the multinomial distribution, and summarizes their properties as a review. The multivariate Krawtchouk polynomials are symmetric functions of orthogonal sets of functions defined on each of $N$ multinomial trials. The dual multivariate Krawtchouk polynomials, which also have a polynomial structure, are seen to occur naturally as spectral orthogonal polynomials in a  Karlin and McGregor spectral representation of transition functions in a composition birth and death process. In this Markov composition  process in continuous time  there are $N$ independent and identically distributed birth and death processes each with support $0,1,\ldots$. The state space in the composition process is the number of processes in the different states $0,1,\ldots$.  Dealing with the spectral representation requires new extensions of the multivariate Krawtchouk polynomials to orthogonal polynomials on a multinomial distribution with a countable infinity of states.
}


\begin{document}
\section{Introduction}
\citet{G1971} and \citet{DG2014} construct multivariate Krawtchouck polynomials orthogonal on the multinomial distribution (\ref{multinomial:0}) and study their properties. Recent representations and derivations of orthogonality of these polynomials are in \citet{GVZ2013,GR2011,I2012,M2011}.

 Authors emphasise different approaches to the multivariate orthogonal polynomials. \citeauthor{DG2014}'s approach is probabilistic and directed to Markov chain applications; \citeauthor{I2012}'s approach is via Lie groups; and \citeauthor{GVZ2013}'s physics approach is as matrix elements of group representations on oscillator states. 
\citet{X2015} studies discrete multivariate orthogonal polynomials which have a triangular construction of products of 1-dimensional orthogonal polynomials. These are particular cases of the polynomials in this paper, see \citet{DG2014}. 
These polynomials extend the Krawtchouk polynomials on the binomial distribution to multi-dimensional polynomials on the multinomial distribution. They appear naturally in composition Markov chains as eigenfunctions in a diagonal expansion of the transition  functions. There are many interesting examples of these Markov chains in \citet{ZL2009}.  Binomial and multinomial random variables can be constructed as a sum of independent identically distributed random variables which are indicator functions of the events that occur on each of $N$ trials. The  Krawtchouk  and multivariate Krawtchouk polynomials are symmetric functions of orthogonal functions sets on each of the trials. The simplest case is the Krawtchouk polynomials where the representation is explained in Section 2.
In the multivariate Krawtchouk polynomials there is not a unique orthogonal function set on trials with multiple outcomes greater than two, so the polynomials depend on which orthogonal function set is taken for a basis on the trials.

 A well known spectral expansion by \citeauthor{KG1957a} 1957a, 1957b, 1958 for the transition functions $\{p_{ij}(t)\}_{i,j=0}^\infty$ of a birth and death process with rates $\lambda_i,\mu_i$, $i=0,1,\ldots$  is that 
\begin{equation}
p_{ij}(t) = \pi_j\int_0^\infty e^{-zt}Q_i(z)Q_j(z)\psi(dz),\>i,j=0,1,\ldots
\label{abc:0}
\end{equation}
where $\{Q_i\}_{i=1}^\infty$ are orthogonal polynomials on the spectral measure $\psi$, which is a probability measure, and
\[
\pi_j = \frac{
\lambda_0\cdots \lambda_{j-1}
}
{\mu_1\cdots \mu_j},\>j=1,2,\ldots
\]
A number of classical birth and death processes have a spectral expansion where the orthogonal polynomials are constructed from the Meixner class.  This class has a generating function of the form
\begin{equation}
G(v,z) = h(v)e^{xu(v)}=\sum_{m=0}^\infty Q_m(z)v^m/m!,
\label{mclass:0}
\end{equation}
where $h(v)$ is a power series in $t$ with $h(0)=1$ and $u(v)$ is a power series with $u(0)=0$ and $u^\prime(0) \ne 0$. \citet{M1934} characterizes the class of weight functions and orthogonal polynomials with the generating function
(\ref{mclass:0}). They include the Krawtchouk polynomials,  Poisson-Charlier polynomials, scaled Meixner polynomials,  and Laguerre polynomials. (The Meixner orthogonal polynomials are a specific set belonging to the Meixner class with a name in common.)

In this paper the spectral expansion is extended to composition birth and death processes where there are $N$ independent and identically distributed birth and death processes  operating and $\{\bm{X}(t)\}_{t\geq 0}$ is such that the $i$th element $\{X_i(t)\}_{t\geq 0}$ counts the number of processes in state $i$ at time $t$.  In the analogue of (\ref{abc:0}) the spectral polynomials are the dual multivariate Krawtchouk polynomials. The dual polynomial system is therefore very important and attention is paid to describing it.  

There are extensions of the multivariate Krawtchouk polynomials to multivariate orthogonal polynomials on the multivariate Meixner distribution and multivariate product Poisson distribution where they occur as eigenfunctions of multi-type birth and death processes \citep{G2016}.

This paper defines the multivariate Krawtchouk polynomials, summarizes their properties, then considers how they are found in spectral expansions of composition birth and death processes.
It is partly a review of these polynomials and is self-contained. For a fuller treatment see \citet{DG2012}. The polynomials are naturally defined by a generating function and so generating functions techniques are used extensively in the paper. Probabilistic notation is used, particularly the expectation operator $\mathbb{E}$ which is a linear operator acting on functions of random variables, which take discrete values in this paper.  If $X_1,\ldots ,X_d$ are random variables then
\[
\mathbb{E}\big [f(X_1,\ldots,X_d)\big ] = \sum_{x_1,\ldots, x_d}f(x_1,\ldots,x_d)
P(X_1=x_1,\ldots,X_d=x_d).
\]
Often orthogonal polynomials are regarded as random variables. For example $\{K_n(X;N,p)\}_{n=0}^N$
are the 1-dimensional Krawthcouk polynomials as random variables and
\begin{eqnarray*}
\mathbb{E}\big [K_n(X;N,p)K_m(X;N,p)\big ]
%\nonumber \\
%&&~=
&=&
\sum_{x=0}^NK_n(x;N,p)K_m(x;N,p){N\choose x}p^xq^{N-x}
\nonumber \\
%&&~
&=& \delta_{mn}{n!}^2{N\choose n}(pq)^n,\>m,n=0,1,\ldots ,N,
\end{eqnarray*}
where $q=1-p$.
A convention of using capital letters for random variables and lower case for values they take is used, except when the random variables are denoted by Greek letters, when they have to be considered in context. 

Section 2, Theorem 1, shows how the Krawtchouk polynomials can be expressed as elementary symmetric functions of $N$ Bernoulli trials, centered at their mean $p$. The Meixner orthogonal polynomials on the geometric distribution are also expressed as functions of an infinity of centered Bernoulli trials in Theorem 2. There is some, but not total symmetry in this expression. Krawtchouk polynomials occur naturally as eigenfunctions in Ehrenfest urn processes and the eigenfunction expansion of their transition functions is explained in Section 2.3. Section 3 introduces the multivariate Krawtchouk polynomials, explaining how they are constructed in a symmetric way from a product set of orthogonal functions on $N$ independent multinomial trials. The dual orthogonal system is described and a scaling found so that they are multivariate Krawtchouk polynomials on a different multinomial distribution in Theorem 3. The polynomial structure of the multivariate Krawtchouk polynomals is described in Theorem 4 and the structure in the dual system in Theorem 5.   Recurrence relationships are found for the system in Theorem 6 and for the dual system in Theorem 7. The dual recurrence relationship is used to identify the polynomials as eigenfunctions in a $d$-type Ehrenfest urn in Theorem 8.  In Section 3.2 a new extension is made to multivariate Krawtchouk polynomials where there are an infinite number of possibilities in each multinomial trial. These polynomials occur naturally as eigenfunctions in composition birth and death processes in a Karlin and McGregor spectral expansion in Theorem 9. Theorem 10 considers the polynomial structure of the dual polynomials in the spectral expansion.  Theorem 11 gives an interesting identity for these spectral polynomials in composition birth and death processes when the spectral polynomials in the individual processes belong to the Meixner class.
\section{Bernoulli trials and orthogonal polynomials}
The paper begins with expressing the 1-dimensional Krawtchouk polynomials as symmetric functions of Bernoulli trials. The multivariate Krawtchouk polynomials are extensions of this construction in higher dimensions. 
\subsection{Krawtchouk orthogonal polynomials}
The Krawtchouk orthogonal polynomials $\{K_n(x;N,p)\}_{n=0}^N$ are orthogonal on the binomial $(N,p)$ distribution 
\[
{N\choose x}p^xq^{N-x},\>x=0,1,\ldots,N.
\]
 They have a generating function 
\begin{equation}
G(z;x) = \sum_{n=0}^NK_n(x;N,p)\frac{z^n}{n!} = \big (1+qz\big )^x\big (1-pz \big)^{N-x}.
\label{Kgenfn:0}
\end{equation}
The scaling is such that the polynomials $K_n(x;N,p)/n!$ are monic and
\begin{eqnarray}
\mathbb{E}\big [K_n(X;N,p)^2\big ] &=& \sum_{x=0}^N{N\choose x}p^xq^{N-x}K_n(x;N,p)^2
\nonumber \\
&=& n!^2{N\choose n}(pq)^n.
\label{hn:0}
\end{eqnarray}
If the Krawtchouk polynomials are scaled to be $Q_n(x)$ so that $Q_n(0) =1$, then
there is a duality that $Q_n(x) = Q_x(n)$. 
A binomial random variable $X$ counts the number of successes in $N$ independent trials, each with a probability $p$ of success. Let $\xi_i=1$ if the $i$th trial is a success and $\xi_i=0$ otherwise. Then $\{\xi_i\}_{i=1}^N$ is a sequence of Bernoulli trials with $P(\xi_i=1) =p$,  $P(\xi_i=0) = q$, and $X=\sum_{i=1}^N\xi_i$. It is interesting to express the Krawtchouk polynomials as symmetric functions of $\{\xi_i\}_{i=1}^N$. If there is just one trial with $N=1$, $X = \xi_1$ and the orthogonal  polynomial set on $X$ is $\{1,\xi_1-p\}$. There can only be a constant function and a linear function if there are just two values that $\xi_1$ can take. A product set of orthogonal functions on $\{\xi_i\}_{i=1}^N$ is $\bigotimes_{i=1}^N \{1,\xi_i-p\}$, and we want to form a smaller basis from  these functions to orthogonal polynomials in $X=\sum_{i=1}^N\xi_i$.
\begin{Theorem}
\label{thm:1}
The Krawtchouk polynomials are proportional to the elementary symmetric functions of $\{\xi_i-p\}_{i=1}^N$;
\begin{equation}
K_n(X;N,p) = n!\sum_{\sigma \in S_N}(\xi_{\sigma(1)}-p)\cdots (\xi_{\sigma(n)}-p),
\label{Ksym:0}
\end{equation}
where $S_N$ is the symmetric group on $\{1,2,\ldots,N\}$.
\end{Theorem}
\begin{proof}
A generating function for the symmetric functions on the right of (\ref{Ksym:0}) is
\begin{equation}
\sum_{n=0}^n\frac{z^n}{n!}n!\sum_{\sigma \in S_N}(\xi_{\sigma(1)}-p)\cdots (\xi_{\sigma(n)}-p)
= \prod_{i=1}^N\big (1 + z(\xi_i -p)\big ).
\label{esfgen:0}
\end{equation}
If $x$ of the $\xi_i$ are 1 and $N-x$ are 0, then (\ref{esfgen:0}) is equal to
\[
(1+z(1-p))^X(1-zp)^{N-X},
\]
identical to the right side of (\ref{Kgenfn:0}). Therefore (\ref{Ksym:0}) holds since the generating functions of both sides, regarding $X$ as a random variable, are the same.
\end{proof}
The representation (\ref{Ksym:0}) of the Krawtchouk polynomials appears with a full treatment in \citet{DG2012} and very briefly as the generating function proof above in \citet{G1971}. The author is not aware of any other appearances of  (\ref{Ksym:0}).
\subsection{Meixner polynomials on the geometric distribution}
The Meixner orthogonal polynomials on the geometric distribution are orthogonal on 
\[
pq^x,\>x=0,1,\ldots
\]
Let $\{\xi_i\}_{i=1}^\infty$ be a sequence of Bernoulli trials.  Let $X$ count  the number of trials $\xi_i=0$ before the first trial where $\xi_{x+1}=1$. That is
$X = \sum_{j=1}^\infty\prod_{i=1}^{j-1}(1-\xi_i)\xi_j.
$
$X$ is clearly not a symmetric function of $\{\xi_i\}_{i=1}^\infty$.  The orthogonal polynomials on the geometric distribution are a special case of the general Meixner polynomials and have a generating function
\[
G(z;x) = \sum_{n=0}^\infty M_n(x;1,q)z^n =  (1-q^{-1}z)^{x}(1-z)^{-(x+1)}.
\]
A product set of orthogonal functions on the trials is  
\begin{equation}
\bigotimes_{i=1}^\infty \{1,\xi_i-p\}.
\label{pset:0}
\end{equation}
It is of interest to express the orthogonal polynomial set $\{M_n(x;1,q)\}_{n=0}^\infty$ as a series expansion in the product set (\ref{pset:0}) as a comparision of what happens with the Krawtchouk polynomials. A calculation is now made of
$\mathbb{E}\big [(\xi_{i_1}-p)\cdots (\xi_{i_r}-p)G(z;X)\big ]$ leading to coefficients in the expansion of the Meixner polynomials in the product set of orthogonal functions.  Given $X=x$ it must be that $\xi_j = 0$, $j=1,\ldots , x$, $\xi_{x+1}=1$ and $\{\xi_j\}_{j=x+2}^\infty$ are distributed as Bernoulli trials. Therefore
\[
\mathbb{E}\big [(\xi_{i_1}-p)\cdots (\xi_{i_r}-p)\mid X=x\big ]
= 
\begin{cases}
0&\mbox{if~} i_r \geq x+2,\\
q(-p)^{r-1}&\mbox{if~} i_r = x+1,\\
(-p)^r&\mbox{if~} i_r \leq x.
\end{cases}
\]
Taking an expectation conditional on $X$, then over $X$,
\begin{eqnarray}
&&\mathbb{E}\big [(\xi_{i_1}-p)\cdots (\xi_{i_r}-p)G(z;X)]
\nonumber \\
&&~= q(-p)^{r-1}G(z;i_r-1)pq^{i_r-1} + \sum_{j=i_r}^\infty(-p)^rG(z;j)pq^j
\nonumber \\
&&~= (-1)^{r-1}p^rq^{i_r}zG(z;i_r-1).
\label{meixprod:0}
\end{eqnarray}
Simplification to the last line is straightforward, so is omitted.
Considering the coefficients of $z^n$ in (\ref{meixprod:0}), and using Theorem \ref{thm:1} gives the following theorem.
\begin{Theorem}
Let $\{\xi_i\}_{i=1}^\infty$ be a sequence of Bernoulli $(p)$ trials and 
$X = \sum_{j=1}^\infty\prod_{i=1}^{j-1}(1-\xi_i)\xi_j$.  Then $X$ has a geometric distribution and the Meixner polynomials on this geometric distribution have a representation for $n\geq 1$ of
\begin{eqnarray}
M_n(X;1,q) &=& \sum_{r=1}^\infty\sum_{i_1<\cdots < i_r}
(\xi_{i_1}-p)\cdots (\xi_{i_r}-p)(-1)^{r-1}q^{i_r-r}M_{n-1}(i_r-1;1,q)
\nonumber \\
&=&
\sum_{r=1}^\infty\sum_{l=r}^\infty
 \frac{1}{(r-1)!}K_{r-1}(X_l;l,p)(\xi_{l}-p)(-1)^{r-1}q^{l-r}M_{n-1}(l-1;1,q),
\label{Mprod:10}
\end{eqnarray}
where $X_l = \xi_1+ \cdots \xi_l$.
\end{Theorem}

\subsection{An Ehrenfest urn}
The Krawtchouk polynomials appear naturally as eigenfunctions in an Ehrenfest urn model. This is explored in \citet{DG2012}. An urn has $N$ balls coloured red or blue. Transitions occur at rate 1 when a ball is chosen at random and the colour of the ball is changed according to a transition matrix
\[
P = \begin{bmatrix}
0&1\\
q/p&1-q/p
\end{bmatrix}
\]
where $p,q > 0$, $p+q=1$ and $q \leq p$. 
Let $\{X(t)\}_{t\geq 0}$  be the number of red balls in the urn at time $t$.
That is, if a blue ball is chosen it is changed to red with probability 1, whereas if a red ball is chosen it is changed to blue with probability $q/p$. 
$\{X(t)\}_{t\geq 0}$  is a reversible Markov process, which is a birth and death process, with a Binomial $(N,p)$ stationary distribution.

The process is a \emph{composition} Markov process in the following sense. Label the balls $1,2,\ldots N$ at time $t=0$ and keep the labels over time as their colours change. Let $\{\xi_i(t)\}_{t\geq 0}$ describe the colour of ball $i$ at time $t$: $\xi_i(t) = 1$ if the $i$th ball is red or 0 if the ball is blue. The processes  $\{\xi_i(t)\}_{t\geq 0}$, $i=1,\ldots,N$ are independent, each has a rate of events $1/N$ when the specified ball is chosen, and $X(t) = \sum_{i=1}^N\xi_i(t)$.  Denote 
$p_{ij}(t) = P(\xi_k(t) = j\mid \xi_k(0)=i)$, for $i,j=1,2$. Standard Markov process theory gives that
\begin{eqnarray}
P(t) = \begin{bmatrix}
q+pe^{-\lambda t}&p(1-e^{-\lambda t})\\
q(1-e^{-\lambda t})&p+qe^{-\lambda t}
\end{bmatrix},
\label{transmatrix:0}
\end{eqnarray}
where $\lambda = 1/(Np)$. It is immediate that the stationary distribution of each of the labelled processes is $(p,q)$. An eigenvalue-eigenfunction expansion of $P(t)$ is
\begin{equation}
P_{\eta,\xi}(t) = \pi_{\xi}\big \{1 + e^{-\lambda t}(pq)^{-1}(\eta-p)(\xi-p)\big \},
\xi,\eta = 0,1,
\label{singleOF:0}
\end{equation}
where $\pi_\xi$ is the stationary distribution with $\pi_0=q$, $\pi_1=p$. It is straightforward to check the agreement with $P(t)$ by substituting the four values of $\eta,\xi=0,1$. 

In the Ehrenfest urn composition process the transitions are made from $X(0)=x$ to $X(t)=y$ if $\sum_{i=1}^N\eta_i=x$ and $\sum_{i=1}^N\xi_i=y$.  The transition probabilities are
\begin{eqnarray}
&&P(X(t)=y\mid X(0)=x)
\nonumber \\
&&~= \sum_{\sigma\in S_N}
P_{\eta_1\xi_{\sigma(1)}}(t)\cdots P_{\eta_N\xi_{\sigma(N)}}(t)
\nonumber \\
&&~=
{N\choose y}p^y(1-p)^{N-y}
\Big \{1 + \sum_{n=1}^Ne^{-\lambda nt}(pq)^{-n}{N\choose n}^{-1}
\nonumber \\
&&~~~\times\sum_{\sigma\in S_N}(\eta_{\sigma(1)}-p)\cdots (\eta_{\sigma(n)}-p)
\sum_{\tau\in S_N}(\xi_{\tau(1)}-p)\cdots (\xi_{\tau(n)}-p)
\Big \}
\nonumber \\
&&~={N\choose y}p^y(1-p)^{N-y}
\nonumber \\
&&~~~\times
\Big \{1 + \sum_{n=1}^Ne^{-\lambda nt}(pq)^{-n}(n!)^{-2}{N\choose n}^{-1}K_n(x;N,p)K_n(y;N,p)\Big\}.
\label{Kpolyexp:0}
\end{eqnarray}
The Krawtchouk polynomials thus appear naturally as elementary symmetric functions of the individual labelled indicator functions in the Markov process.

\section{Multivariate Krawtchouk polynomials}\label{section:MVK}
The multivariate Krawtchouk polynomials with elementary basis $\bm{u}$ were first constructed by  \citet{G1971}. A recent introduction to them is \citet{DG2014}. They play an important role in the spectral expansion of transition functions of composition Markov processes. \citet{ZL2009,KZ2009} have many interesting examples of such Markov processes. Later in this paper we consider the particular composition processes where there are $N$ particles independently performing birth and death processes.

The multivariate Krawtchouk polynomials are orthogonal on the multinomial distribution
\begin{equation}
m(\bm{x};\bm{p}) = {N\choose \bm{x}}\prod_{j=1}^dp_j^{x_j},\>x_j \geq 0, \>j=1,\ldots, d,\> |\bm{x}|=N,
\label{multinomial:0}
\end{equation}
with $\bm{p}=\{p_j\}_{j=1}^d$ a probability distribution.
Let $J_1,\ldots ,J_N$ be independent identically distributed random
variables specifying outcomes on the $N$ trials such that
\[
P(J=k) = p_k, \>k=1,\ldots ,d.
\]
Then 
\[
X_i = |\{J_k: J_k=i, k=1,\ldots, N\}|.
\]
Let $\bm{u}=\{u^{(l)}\}_{l=0}^{d-1}$ be an orthogonal set of functions on $\bm{p}=\{p_k\}_{k=1}^d$ with $u^{(0)}=1$ satisfing 
\begin{equation}
\sum_{i=1}^{d}u^{(l)}_iu^{(m)}_ip_i = a_l\delta_{lm},\>l,m = 0,\ldots d-1.
\label{basis:0}
\end{equation}
This notation for the orthogonal set of functions  follows \citet{L1969}. There is an equivalence that 
\[
h_{il}=u^{(l-1)}_i\sqrt{p_i/a_{l-1}},\>\>i,l=1,\ldots,d
\]
are elements of a $d\times d$ orthogonal matrix $H$. In this paper $\{u^{(l)}\}_{l=0}^{d-1}$ are usually orthonormal functions with $a_l=1$, $l=0,1,\ldots, d-1$ unless stated otherwise.
The 1-dimensional Krawtchouk polynomials are constructed from a symmetrized product set of orthogonal functions  $\bigotimes \{1,\xi_i-p\}_{i=1}^N$ and the construction of the multivariate polynomials follows a similar, but more complicated proceedure. Instead of having two unique elements in each orthogonal function set there is a choice of orthogonal basis and the construction is from the product set 
$\bigotimes_{i=1}^N \{u^{(l_i)}_{J_i}\}_{l_i=0}^{d-1}$. The orthogonality (\ref{basis:0}) is equivalent to  
\[
\mathbb{E}\big [u^{(l)}_{J_k}u^{(m)}_{J_k}\big ] = a_l\delta_{lm},
\]
for $k=1,\ldots,N$.
Define a collection of orthogonal polynomials
$\big \{Q_{\bm{n}}(\bm{X};\bm{u})\big \}$ with $\bm{n} = (n_1,\ldots n_{d-1})$
and $|\bm{n}| \leq N$ on the multinomial distribution as symmetrized elements from the product set such that the sum is over products 
$
u_{J_1}^{(l_1)}\cdots u_{J_N}^{(l_N)}
$
with $n_k = |\{l_i:l_i=k,k=1,\ldots ,N\}|$
for $k=1,\ldots,d-1$.
$Q_{\bm{n}}(\bm{X};\bm{u})$ is the coefficient of
 $w_1^{n_1}\cdots w_{d-1}^{n_{d-1}}$ in the generating function
\begin{eqnarray}
G(\bm{x},\bm{w}, \bm{u}) 
&=&
\prod_{i=1}^N\Big (1 + \sum_{l_i=1}^{d-1}w_{l_i}u_{J_i}^{(l_i)}\Big )
\nonumber \\
&=&
 \prod_{j=1}^d\Big (1 + \sum_{l=1}^{d-1}w_lu_j^{(l)}\Big )^{x_j}.
\label{main_gf}
\end{eqnarray}
In the 1-dimensional case $u^{(1)}_1=0-p_1=-p_1$,  $u^{(1)}_2=1-p_1$, orthogonal on $1-p_1,p_1$, so the generating function is
\[
(1-p_1w_1)^{x_1}(1+(1-p_1)w_1)^{x_2}
\]
which is, of course, the generating function of the Krawtchouk polynomials. $x_1,x_2$ are respectively the number of 0 and 1 values in the $N$ trials.
It is straightfoward to show, by using the generating function (\ref{main_gf}), that
\begin{eqnarray}
\mathbb{E}\Big [Q_{\bm{m}}(\bm{X}; \bm{u})Q_{\bm{n}}(\bm{X};\bm{u})\Big ] &=&
\sum_{\{\bm{x}:|\bm{x}|=N\}}Q_{\bm{m}}(\bm{x}; \bm{u})Q_{\bm{n}}(\bm{x};\bm{u})
m(\bm{x};\bm{p})
\nonumber \\
&=&\delta_{\bm{m}\bm{n}}{N\phantom{^+}\choose \bm{n}^+}\prod_{j=1}^{d-1}{a_j}^{n_j},
\label{normalizing}
\end{eqnarray}
where  $\bm{n}^+= (n_0,\ldots ,n_{d-1})$, with $n_0=N-\sum_{j=1}^{d-1}n_j$.
Instead of indexing the polynomials by $\bm{n}=(n_1,\ldots ,n_{d-1})$ they could be indexed by $\bm{n}^+$. This notation is sometimes convenient to use in the paper. 
The dual orthogonality relationship is, immediately from (\ref{normalizing}),
\begin{equation}
\sum_{\{\bm{n}:|\bm{n}| \leq N\}}{N\phantom{^+}\choose \bm{n}^+}^{-1}\prod_{j=1}^{d-1}{a_j}^{-n_j}Q_{\bm{n}}(\bm{x};\bm{u})Q_{\bm{n}}(\bm{y};\bm{u}) = \delta_{\bm{x}\bm{y}}m(\bm{x},\bm{p})^{-1}.
\label{dualnorm:0}
\end{equation} 
Expanding the generating function (\ref{main_gf}) shows that
\begin{equation}
 Q_{\bm{n}}(\bm{X};\bm{u})= \sum_{\{\bm{r}:r_{\cdot k} = n_k\}}
\frac{ \prod_{j=1}^d{x_j}_{[r_j\cdot]} !}
{ \prod_{j=1}^d\prod_{k=1}^{d-1}r_{jk}! }
 \prod_{j=1}^d\prod_{k=1}^{d-1}\Big (u_j^{(k)}\Big )^{r_{jk}},
\label{explicit}
\end{equation}
where $\cdot$ indicates summation over an index and $a_{[b]} = a(a-1)\cdots (a-b+1)$ for non-negative integers $b$. 
The dual generating function is
\begin{eqnarray}
&&\sum_{\{\bm{x}:|\bm{x}|=N\}}
{N\phantom{^+}\choose \bm{n}^+}^{-1}
{N\choose \bm{x}}
v_1^{x_1}\cdots v_d^{x_d}
Q_{\bm{n}}(\bm{x};\bm{u})
\nonumber \\
&&~~=
\Big (\sum_{j=1}^dv_j\Big )^{n_0}
\prod_{i=1}^{d-1}\Big (\sum_{j=1}^{d}v_ju^{(i)}_j\Big )^{n_i}.
\label{dualgf:0}
\end{eqnarray}
Expanding the generating function
\begin{equation}
{N\phantom{^+}\choose \bm{n}^+}^{-1}
{N\choose \bm{x}}
Q_{\bm{n}}(\bm{x};\bm{u})
=
\sum_{\{\bm{r}:r_{i\cdot}=n_i,r_{\cdot j}=x_j\}}
\frac{\prod_{i=0}^{d-1}n_i!}{\prod_{i=0}^{d-1}\prod_{j=1}^dr_{ij}!}
\prod_{i=1}^{d-1}\prod_{j=1}^d\Big (u_j^{(i)}\Big )^{r_{ij}}.
\label{matrixh:1}
\end{equation}
The two generating functions (\ref{main_gf}) and (\ref{dualgf:0}) are similar and there is a form of self-duality for the polynomials. Let 
\[
\omega_i^{(j)}=u_{j+1}^{(i-1)},\>j=0,\ldots ,d-1,i=1,\ldots,d.
\]
Then because of (\ref{basis:0}) 
\[
\sum_{l=1}^d\omega_l^{(i)}\omega_l^{(k)}a_{l-1}^{-1}= \delta_{ik}p_i^{-1}.
\]
%Let $\{b_l\}$ be the scaled probability measure constructed from $\{a_{l-1}^{-1}\}$.
The right side of (\ref{dualgf:0}) is equal to
\begin{equation}
\prod_{i=1}^{d}
\Big (\sum_{j=0}^{d-1}
\omega_i^{(j)}v_{j+1}\Big )^{n_{i-1}}
\label{dualx:10}
\end{equation}
which apart from the different indexing, and non-constant function $\omega^{(0)}$, generates multivariate Krawtchouk polynomials. Suppose that $\omega^{(0)}_i \ne 0$ for $i=1,\ldots ,d$. Scale by letting $\widehat{\omega}_i^{(j)} = \omega_i^{(j)}/\omega_i^{(0)}$, so that $\widehat{\omega}_i^{(0)}=1$. Orthogonality of these functions is
\[
\sum_{l=1}^d\widehat{\omega}_l^{(i)} \widehat{\omega}_l^{(j)} a_{l-1}^{-1}{\omega_l^{(0)}}^2 = \delta_{ij}p_i^{-1}.
\]
Let $\bm{b}=\{b_l\}_{l=1}^d$ be the scaled probability distribution of $\{ a_{l-1}^{-1}{\omega_l^{(0)}}^2\}_{l=1}^d$ so
\[
\sum_{l=1}^d\widehat{\omega}_l^{(i)} \widehat{\omega}_l^{(j)}b_l 
 = \delta_{ij}\Big (p_i\sum_{l=1}^da_{l-1}^{-1}{\omega_l^{(0)}}^2\Big )^{-1}.
 \]
The following theorem is evident from (\ref{dualgf:0}) and (\ref{dualx:10}), once the indexing is sorted out.
\begin{Theorem}\label{thm:2}
There is a duality
\begin{equation}
{N\phantom{^+}\choose \bm{n}^+}^{-1}
{N\choose \bm{x}}
Q_{\bm{n}}(\bm{x};\bm{u}) 
= \prod_{i=1}^{d}\big ({\omega_i^{(0)}}\big )^{n_{i-1}}Q^*_{\bm{x}^-}(\bm{n}^+;\widehat{\bm{\omega}}),
\label{dualp:10}
\end{equation}
where $Q^*_{\bm{x}^-}(\bm{n}^+;\widehat{\bm{\omega}})$, with $\bm{x}^- = (x_2,\ldots ,x_d)$, $\sum_{j=2}^dx_j \leq N$, are multivariate Krawtchouk polynomials, orthogonal on
 $m(\bm{n}^+;\bm{b})$.
\end{Theorem}
There is an interesting identity when $\bm{u}$ is self-dual with an indexing of $j$ beginning from $0$ instead of $1$. That is
\[
u_j^{(l)} = u_l^{(j)},\>\>j,l=0,1,\ldots ,n.
\]
Then indexing $\bm{x}=(x_0,\ldots,x_n)$,
\[
{N\phantom{^+}\choose \bm{n}^+}^{-1}Q_{\bm{n}}(\bm{x};\bm{u}) =
{N\choose \bm{x}}^{-1}Q_{\bm{x}}(\bm{n}^+;\bm{u}^*),
\]
where ${u_j^{(l)}}^* = u_l^{(j)}$. This duality occurs in the scaled Krawtchouk polynomial basis, orthogonal on a Binomial $(n,p)$ distribution.

The emphasis in Theorem \ref{thm:2} is on considering the dual system, obtaining $\widehat{\bm{\omega}}$ from $\bm{u}$,
 however sometimes it is natural to construct $\bm{u}$ from an orthogonal set
 $\widehat{\bm{\omega}}$, particularly when $\omega_i^{(0)}=1$, $i=1,\ldots,d$ and
 $\widehat{\bm{\omega}}=\bm{\omega}$. Then the polynomials on the left of (\ref{dualp:10}) are defined by the dual polynomials on the right. Later in the paper it will be seen that this is natural in composition birth and death Markov processes.

The polynomial structure of the multivariate Krawtchouk polynomials is detailed in the next theorem.
\begin{Theorem}\label{thm:3}
Define
$U_l = \sum_{k=1}^Nu_{J_k}^{(l)} = \sum_{j=1}^du_j^{(l)}X_j$
for $l =1,\ldots ,d-1$.
$Q_{\bm{n}}(\bm{X};\bm{u})$ is a polynomial of degree $|\bm{n}|$ in $(U_1,\ldots ,U_{d-1})$
whose only term of maximal degree $|\bm{n}|$ is $\prod_1^{d-1}U_k^{n_k}$.
\end{Theorem}
\begin{proof}
A method of proof is to consider the transform of 
$Q_{\bm{n}}(\bm{X};\bm{u})$, which is given by 
\begin{equation}
E\Big [\prod_{j=1}^d\phi^{X_j}_jQ_{\bm{n}}(\bm{X};\bm{u})\Big ] = 
{N\choose |\bm{n}|}{|\bm{n}|\choose \bm{n}}T_0(\phi)^{N-|n|}T_1(\bm{\phi})^{n_1}\cdots T_{d-1}(\bm{\phi})^{n_{d-1}},
\label{transform}
\end{equation}
where 
\[
T_i(\bm{\phi}) = \sum_{j=1}^{d}p_j\phi_ju_j^{(i)},\>i=0,\ldots ,d-1.
\]
This transform is easily found by taking the transform of the generating function (\ref{main_gf}).
One can see directly that $Q_{\bm{n}}(\bm{X};\bm{u})$ is an orthogonal polynomial by considering the transform
\begin{equation}
E\Big (\prod_{j=1}^d{X_j}_{[k_j]}\phi_j^{X_j}\Big )
= N_{[k]}\prod_{j=1}^d(\phi_jp_j)^{k_j}\cdot \big (\sum_{j=1}^dp_j\phi_i\big )^{N-|k|}.
\label{tr_one}
\end{equation}
From (\ref{transform}) and (\ref{tr_one}), $Q_n(x)$ is a polynomial of degree $|n|$ whose only leading term is 
\[
\frac{\prod_{i=1}^{d-1}S_i^{n_i}}{\prod_{i=1}^{d-1}n_i!}.
\]
This is seen by noting that the leading term is found by replacing $\phi_jp_j$ 
by $X_j$ in
\[
N_{[|n|]}^{-1}{N\choose |n|}{|n|\choose n}T_1(\bm{\phi})^{n_1}\cdots T_{d-1}(\bm{\phi})^{n_{d-1}},
\]
since we can replace ${X_j}_{[k_j]}$ by $X^{k_j}$ in considering the leading term of
(\ref{transform}) and setting $\phi_i=1$ for $i=1,\ldots, d$.
\end{proof}
The next theorem explains the polynomial structure in the dual system.
\begin{Theorem}\label{thm:dual}
Let $\{u^{(j)}\}_{j=0}^d$ be such that $u^{(j)}_1=1$ for $j=0,\ldots,d-1$ as well as the usual assumption that $u^{(0)}_i=1$ for $i=1,\ldots, d$. 
Define $\kappa_l=\sum_{j=0}^{d-1}u^{(j)}_ln_j$, $l=2,\ldots,d$. Then
$
{N\phantom{^+}\choose \bm{n}^+}^{-1}
Q_{\bm{n}}(\bm{x};\bm{u}) 
$
is a polynomial of total degree $\sum_{i=2}^dx_i$ in $\{\kappa_l\}_{l=2}^{d}$ whose only term of maximal degree is $\prod_{l=2}^d\kappa_j^{x_j}$.
\end{Theorem}
\begin{proof}
This follows from Theorem \ref{thm:2}, with $\omega_i^{(0)}=1$, $i=1,\ldots d$, and Theorem \ref{thm:3}.
\end{proof}
There are recurrence relationships for the multivariate Krawtchouk polynomials, which are found here from a generating function approach; for another different proof see \citet{I2012}, Theorem 6.1. Note that his multivariate Krawtchouk polynomials are 
$Q_{\bm{n}}(\bm{x};\bm{u}){N\phantom{^+}\choose \bm{n}^+}^{-1}$.
 In Theorems 6, 7, 8, $\bm{u}$ is taken to be orthonormal on $\bm{p}$, so $a_l=1$, $l=0,1,\ldots d-1$ in (\ref{basis:0}).
\begin{Theorem}
Denote, for $i,l,k=0,\ldots,d-1$,
$
c(i,l,k) = \sum_{j=1}^du_j^{(i)}u_j^{(l)}u_j^{(k)}p_j,
$
and $u_i = \sum_{j=1}^du_j^{(i)}x_j$, $i=1,\ldots, d-1$.
Two recursive systems are:
\begin{eqnarray}
x_jQ_{\bm{n}}(\bm{x};\bm{u}) &=& \sum_{k=1}^{d-1}(n_k+1)p_ju_j^{(k)}Q_{\bm{n}+\bm{e}_k}(\bm{x};\bm{u})
\nonumber \\
&&~~~~~+ (N-|\bm{n}|+1)\sum_{l=1}^{d-1}p_ju^{(l)}_jQ_{\bm{n}-\bm{e}_l}(\bm{x},\bm{u})
\nonumber \\
&&~~~~~
+\sum_{l,k=1}^{d-1}(n_k+1-\delta_{lk})p_ju^{(l)}_ju_j^{(k)}Q_{\bm{n}-\bm{e}_l+\bm{e}_k}(\bm{x};\bm{u}) + p_j(N-|\bm{n}|)Q_{\bm{n}}(\bm{x};\bm{u})
\label{recone:0}
\end{eqnarray}
and
\begin{eqnarray}
u_iQ_{\bm{n}}(\bm{x};\bm{u}) &=& (n_i+1)Q_{\bm{n}+\bm{e}_i}(\bm{x};\bm{u}) + (N-|\bm{n}|+1)
%n_i^{-1}
Q_{\bm{n}-\bm{e}_i}(\bm{x},\bm{u})
\nonumber \\
&&~
+\sum_{l,k=1}^{d-1}c(i,l,k)(n_k+1-\delta_{kl})Q_{\bm{n}-\bm{e}_l+\bm{e}_k}(\bm{x};\bm{u}).~~~~~~~
\label{recone}
\end{eqnarray}
\end{Theorem}
\begin{proof}
Consider
\[
\mathbb{E}\Big [X_jG(\bm{X},\bm{w},\bm{u})G(\bm{X},\bm{z},\bm{u})\Big ]
= Np_j
\Big (1 + \sum_{i=1}^{d-1}w_iu_j^{(i)}\Big )
\Big (1 + \sum_{i=1}^{d-1}z_iu_j^{(i)}\Big )
\Big (1 + \sum_{i=1}^{d-1}w_iz_i\Big )^{N-1}.
\]
Equating coefficients of 
$
\prod_1^{d-1}w_j^{n_j}\prod_1^{d-1}z_j^{n_j^\prime}
$;
\begin{equation}
\mathbb{E}\Big [X_jQ_{\bm{n}}(\bm{X};\bm{u})Q_{\bm{n}^\prime}(\bm{X};\bm{u})\Big ]
= 
\begin{cases}
\frac{N!}{(N-|\bm{n}|-1)!\prod_1^{d-1}n_i!}p_ju_j^{(k)}&\bm{n}^\prime = \bm{n} + \bm{e}_k\\
\frac{N!}{(N-|\bm{n}|)!\prod_1^{d-1}n_i!}n_lp_ju_j^{(l)}&\bm{n}^\prime = \bm{n} - \bm{e}_l\\
\frac{N!}{(N-|\bm{n}|)!\prod_1^{d-1}n_i!}n_lp_ju_j^{(l)}u_j^{(k)}&\bm{n}^\prime = \bm{n} - \bm{e}_l + \bm{e}_k,\> l \ne k\\
\frac{N!}{(N-|\bm{n}|-1)!\prod_1^{d-1}n_i!}p_j +
\sum_{l=1}^{d-1}\frac{N!}{(N-|\bm{n}|)!\prod_1^{d-1}n_i!}n_lp_j{u_j^{(l)}}^2&\bm{n}^\prime = \bm{n} .
\end{cases}
\label{cases:0}
\end{equation}
The first recursive equation (\ref{recone:0}) then follows by an expansion of $x_jQ_{\bm{n}}(\bm{x};\bm{u})$ as a series in $Q_{\bm{n}^\prime}(\bm{x};\bm{u})$ dividing the cases in  (\ref{cases:0})  to obtain the coefficients by 
\[
{N\phantom{^+}\choose {\bm{n}^\prime}^+} =
\begin{cases}
\frac{N!}{(N-|\bm{n}|-1)!\prod_1^{d-1}(n_i+\delta_{ik})!}&\bm{n}^\prime = \bm{n} + \bm{e}_k\\
\frac{N!}{(N-|\bm{n}|+1)!\prod_1^{d-1}(n_i-\delta_{il})!}&\bm{n}^\prime = \bm{n} - \bm{e}_l\\
\frac{N!}{(N-|\bm{n}|)!\prod_1^{d-1}(n_i-\delta_{il}+\delta_{ik})!}&\bm{n}^\prime = \bm{n} - \bm{e}_l + \bm{e}_k,\> l \ne k\\
\frac{N!}{(N-|\bm{n}|)!\prod_1^{d-1}n_i!}&\bm{n}^\prime = \bm{n} .
\end{cases}
\]
The second recursion (\ref{recone}) is found by summation, using the orthogonality of $\bm{u}$.
\end{proof}
%%%%%%%%%%%%%%%%%%%%
%%%%%%%%%%%%%%%%%%%%
The dual orthogonal system when $\bm{u}$ is orthonormal is
\begin{equation}
\sum_{\{\bm{n}\geq 0:|\bm{n}| = N\}}
Q_{\bm{n}^+}(\bm{x};\bm{u})Q_{\bm{n}^+}(\bm{x};\bm{u}){N\phantom{^+}\choose \bm{n}^+}^{-1} =
m(\bm{x},\bm{p})^{-1}\delta_{\bm{x}\bm{y}}.
\label{dualMVK:0}
\end{equation}
A dual generating function is 
%${N\choose \bm{x}}Q_{\bm{n}}(\bm{x};\bm{u})$ is the coefficient of $v_1^{x_1}\cdots v_d^{x_d}$ in 
\begin{equation}
%{N\choose \bm{n}, N - |\bm{n}|}\Bigg (\sum_{j=1}^dv_j\Bigg )^{N - |\bm{n}|}
%\cdot 
H(\bm{n},\bm{v},\bm{u})
 = \sum_{\{\bm{x}:|\bm{x}|=N\}}{N\phantom{^+}\choose \bm{n}^+}^{-1}{N\choose \bm{x}}Q_{\bm{n}}(\bm{x};\bm{u})\prod_{i=1}^dv_i^{x_i}
=\prod_{l=0}^{d-1}\Bigg (\sum_{j=1}^du_j^{(l)}v_j\Bigg )^{n_l}.
\label{dualGF:0}
\end{equation}
The generating function (\ref{dualGF:0}) arises from considering the coefficient of $\prod_{i=1}^{d-1}w_i^{n_i}$ in 
\[
\sum_{\{\bm{x}:|\bm{x}|=N\}}{N\phantom{^+}\choose \bm{n}^+}^{-1}{N\choose \bm{x}}G(\bm{x},\bm{w},\bm{u})\prod_{i=1}^dv_i^{x_i}
=
{N\phantom{^+}\choose \bm{n}^+}^{-1}
\Bigg (\sum_{j=1}^{d-1}v_j + \sum_{l=1}^{d-1}w_l\sum_{j=1}^dv_ju_j^{(l)}\Bigg )^N.
\]
\begin{Theorem}\label{thm:transform}
A dual recurrence system is, for $i=0,\ldots ,d-1$
\begin{equation}
n_iQ_{\bm{n}}(\bm{x};\bm{u}) = \sum_{j,l=1}^dx_ju_j^{(i)}u_l^{(i)}p_lQ_{\bm{n}}(\bm{x}-\bm{e}_j+\bm{e}_l;\bm{u})
\label{dualrec:0}
\end{equation}
\end{Theorem}
\begin{proof}
A derivation of the recurrence system uses a transform method. Consider 
\begin{eqnarray*}
&&\sum_{\{\bm{n}^+:|\bm{n}^+|=N\}}n_i\mathbb{E}\Big [\prod_{i=1}^d\phi_i^{X_i}\varphi_i^{Y_i}Q_{\bm{n}}(\bm{X};\bm{u})Q_{\bm{n}}(\bm{Y};\bm{u})\Big ]{N\phantom{^+}\choose \bm{n}^+}^{-1}
\nonumber \\
&&~~
= NT_i(\bm{\phi})T_i(\bm{\varphi})\Big [\sum_{j=1}^dp_j\phi_j\varphi_j\Big ]^{N-1}
\end{eqnarray*}
Therefore non-zero terms with $\bm{y}=\bm{x}-\bm{e}_j+\bm{e}_l$ are
\begin{eqnarray}
\sum_{\{\bm{n}^+:|\bm{n}^+|=N\}}n_iQ_{\bm{n}}(\bm{x};\bm{u})Q_{\bm{n}}(\bm{y};\bm{u})
{N\phantom{^+}\choose \bm{n}^+}^{-1}
&=& \frac{
N{N-1\choose \bm{x}-\bm{e}_j}\prod_{k=1}^dp_k^{x_k-\delta_{jk}}p_ju_j^{(i)}p_lu_l^{(i)}
}
{m(\bm{x},\bm{p})m(\bm{y},\bm{p})}
\nonumber \\
&=& \frac{
x_ju_j^{(i)}p_lu_l^{(i)}
}
{m(\bm{y},\bm{p}) }.
\end{eqnarray}
The dual recurrence is therefore (\ref{dualrec:0}).
\end{proof}
The reproducing kernel polynomials
\[
Q_n(\bm{x},\bm{y}) = \sum_{\{\bm{n}:|\bm{n}|=n\}}{N\choose \bm{n}}^{-1}
Q_{\bm{n}}(\bm{x};\bm{u})Q_{\bm{n}}(\bm{y};\bm{u})
\]
are invariant under which set of orthonormal functions $\bm{u}$ is used. They have an explicit form, see \citet{DG2014} and \citet{X2015} for details.

\subsection{An Ehrenfest urn with $d$-types}
A $d$-type Ehrenfest urn has $N$ balls of $d$ colours $\{1,\ldots ,d\}$. At rate 1 a ball is chosen and if it is of type $j$ it is changed to colour $l$ with probability $p_{jl}$, $l=1,\ldots,d$. $\{\bm{X}(t)\}_{t\geq 0}$, with $|\bm{X}(t)|=N$, is the number of balls of the different colours at time $t$, which can be regarded as a $d$-dimensional random walk on $|\bm{x}|=N$. The transition functions have an eigenfunction expansion in the multivariate Krawtchouk polynomials, extending  the case (\ref{Kpolyexp:0}) with two colours.
\begin{Theorem}\label{thm:walk}
Let $\{\bm{X}(t)\}_{t\geq 0}$ be a $d$-dimensional random walk on $\bm{x}$,
$|\bm{x}|=N$, where transitions are made from $\bm{x} \to \bm{x} - \bm{e}_j + \bm{e}_l$ at rate $r( \bm{x}, \bm{x} - \bm{e}_j + \bm{e}_l)=(x_j/N)p_{jl}$. $P$ is a $d\times d$ transition matrix, with stationary distribution $\bm{p}$ such that
\[
p_{jl} = p_l\Big \{1 + \sum_{i=1}^{d-1}\rho_iu_j^{(i)}u_l^{(i)}\Big \}.
\]
Then the transition functions of $\bm{X}(t)$ have an eigenfunction expansion
\begin{eqnarray}
&&p(\bm{x},\bm{y};t) = m(\bm{y},\bm{p})~~~~~~
\nonumber \\
&&\times
\Bigg \{1 + \sum_{\{\bm{n}:0 < |\bm{n}| \leq N\}}e^{-t \sum_{i=1}^{d-1}n_i(1-\rho_i)/N}{N\choose \bm{n}}^{-1}Q_{\bm{n}}(\bm{x};\bm{u})Q_{\bm{n}}(\bm{y};\bm{u})\Bigg \}.~~~~
\label{transitionfns:0}
\end{eqnarray}
\end{Theorem}
\begin{proof}
 $\{\bm{X}(t)\}_{t\geq 0}$ is a reversible Markov process with stationary distribution $m(\bm{x};\bm{p})$ becuse it satisfies the balance equation
\[
m(\bm{x};\bm{p})r(\bm{x},\bm{x}-\bm{e}_j+\bm{e}_l) =
m(\bm{x}-\bm{e}_j+\bm{e}_l;\bm{p})r(\bm{x}-\bm{e}_j+\bm{e}_l,\bm{x}).
\]
The reversibility is a consequence of assuming that $P$ is a reversible transition matrix.
The generator of the process acting on $f(\bm{x})$ is specified by 
\[
Lf(\bm{x}) = \sum_{j,l}r( \bm{x}, \bm{x} - \bm{e}_j + \bm{e}_l)\big (f( \bm{x} - \bm{e}_j + \bm{e}_l) - f(\bm{x})\big )
\]
so the eigenvalues and eigenvectors $(\lambda_{\bm{n}},g_{\bm{n}}(\bm{x}))$ satisfy
\begin{equation}
Lg_{\bm{n}}(\bm{x}) = -\lambda_{\bm{n}}g_{\bm{n}}(\bm{x}).
\label{generalev:0}
\end{equation}
Now from (\ref{dualrec:0}) 
\begin{equation*}
-\sum_{i=1}^{d-1}\big (n_i(1-\rho_i)/N\big )Q_{\bm{n}}(\bm{x};\bm{u}) = \sum_{j,l=1}^d(x_j/N)p_{jl}Q_{\bm{n}}(\bm{x}-\bm{e}_j+\bm{e}_l;\bm{u})
- Q_{\bm{n}}(\bm{x};\bm{u}) 
\end{equation*}
which is the same as (\ref{generalev:0}), noting that the total rate is 1 away from $\bm{x}$. Then (\ref{transitionfns:0}) holds immediately.
\end{proof}

\subsection{Extensions to the multivariate Krawtchouk polynomials}
It is useful in considering spectral expansions of composition Markov processes to allow the following generalizations of the  multivariate Krawtchouk polynomials.
\begin{itemize}
\item
Allow $d=\infty$ as a possibility and let $\{u^{(j)}\}_{j=0}^\infty$ be a complete orthogonal set of functions on $p_1,p_2,\ldots $. The multinomial distribution is still well defined as
\[
m(\bm{x};\bm{p}) = \frac{N!}{x_1!x_2!\cdots}p_1^{x_1}p_2^{x_2}\cdots
,\> |\bm{x}|=N,
\]
and the generating function for the multivariate Krawtchouk polynomials still holds with $d=\infty$.
\item
When $d=\infty$ take $\{u^{(j)}\}_{j=1}^\infty$ to be orthogonal on a discrete measure $\bm{\pi}$  which is non-negative, but not a probability measure because $\sum_{i=1}^\infty \pi_i = \infty$.
\item
Allow the basis functions $\bm{u}$ to be orthogonal on $\bm{\pi}$, and take the dual functions 
$\{u_i^{(z)}\}_{i=0}^\infty$ to be orthogonal on a continuous distribution. An example that occurs naturally in composition birth and death chains is when $u_i^{(z)} = L_i^{(\alpha)}(z)$, $z \geq 0$, $i=0, 1 ,\ldots$ are the Laguerre polynomials, orthogonal on the density
\[
\frac{z^\alpha}{\Gamma(\alpha+1)}e^{-z},\> z >0.
\]
\end{itemize}
%%%%%%%%%%%%%%%%%%%%%%%%
\subsection{Karlin and McGregor spectral theory}
Consider a birth and death process $\{\xi(t)\}_{t\geq 0}$ on $\{-1,0,1,\ldots \}$ with birth and death rates  $ \lambda_i,\mu_i$ from state $i$ and transition probabilities $p_{ij}(t)$.
$-1$ is an absorbing state which can be reached if $\mu_0 > 0$. We assume that the process is non-explosive so only a finite number of events will take place in any finite time interval.
% In this paper we only consider non-absorbing processes where $\mu_0=0$.
 Define orthogonal polynomials $\{Q_n(z)\}_{n=0}^\infty$ by
\begin{equation}
-zQ_n(z) = -(\lambda_n+\mu_n)Q_n(z) + \lambda_nQ_{n+1}(z) + \mu_nQ_{n-1}(z)
\label{OPSpectrum}
\end{equation}
for $n\in \mathbb{Z}_+$ with $Q_0=1$ and $Q_{-1}=0$.
The polynomials are defined by recursion from (\ref{OPSpectrum}) with $Q_{n+1}$ defined by knowing $Q_n$ and $Q_{n-1}$. If $\mu_0=0$, then $Q_n(0)=1$.
There is a spectral measure $\psi$ with support on the non-negative axis and total mass 1 so that
\begin{equation}
p_{ij}(t) = \pi_j\int_0^\infty e^{-z t}Q_i(z)Q_j(z)\psi(dz),
\label{KMG:0}
\end{equation}
for $i,j=0,1,\ldots$ where 
\[
\pi_j = \frac{\lambda_0\cdots \lambda_{j-1}}{\mu_1\cdots \mu_j}.
\]
If $\mu_0 > 0$ then $\sum_{j=0}^\infty p_{ij}(t) < 1$ because of possible absorption into state $-1$.  If $\mu_0 = 0$ but there is no stationary distribution because $\sum_{j=0}^\infty\pi_j = \infty$ then also possibly  $\sum_{j=0}^\infty p_{ij}(t) < 1$.
Placing $t=0$ shows the orthogonality of the polynomials $\{Q_i(z)\}_{i\geq 0}$ on the measure $\psi$ because $p_{ij}(0) = \delta_{ij}$. $\{\xi(t)\}_{t\geq 0}$ is clearly reversible with respect to $\{\pi_j\}_{j\geq 0}$ when a stationary distribution exists, or before absorption at 0 if it does not exist since
$\pi_ip_{ij}(t) = \pi_jp_{ji}(t)$.
As $t \to \infty$ the limit stationary distribution, if $\mu_0=0$ and $\sum_{k=0}^\infty \pi_k < \infty$, is
\begin{equation}
p_j = \frac{\pi_j}{\sum_{k=0}^\infty \pi_k}=\pi_j\psi(\{0\}).
\label{stationary:0}
\end{equation}
Suppose a stationary distribution exists and there is a discrete spectrum with support $\{\zeta_l\}_{l\geq 0}$, $\zeta_0=0$. 
Then
\begin{eqnarray}
p_{ij}(t) &=& \pi_j\sum_{l=0}^\infty e^{-\zeta_lt}Q_i(\zeta_l)Q_j(\zeta_l)\psi(\{\zeta_l\})
\nonumber\\
&=& 
 p_j\Big \{1 + \sum_{l=1}^\infty e^{-\zeta_lt}Q_i(\zeta_l)Q_j(\zeta_l)\psi(\{\zeta_l\})/\psi(\{0\})\Big \}.
\label{discrete:s0}
\end{eqnarray}
This is an eigenfunction expansion
\begin{equation}
p_{ij}(t) = p_j\Big \{1 + \sum_{l\geq 1}e^{-\zeta_lt}u_i^{(l)}u_j^{(l)}\Big \},\>i,j=0,1\ldots
\end{equation}
where $\bm{u}$ is a set of orthonormal functions on $\bm{p}$ defined by
\[
u^{(l)}_i=Q_i(\zeta_l)\sqrt{\psi(\{\zeta_l\})/\psi(\{0\})},\>i,l=0,1,\ldots
\]
Several well known birth and death processes give rise to classical orthogonal polynomial systems.
In this paper only processes where $\mu_0=0$ are considered so there is no absorbing state at $-1$ and the state space is $\{0,1,\ldots\}$. Classical papers where theory is developed and particular spectral expansions are derived are \citeauthor{KG1957a} 1957a, 1957b, 1958, 1958, 1965.
 \citet{S2000} details the birth and death processes and spectral expansions nicely, from which we summarize.
\begin{itemize}
\item
\emph{The $M/M/\infty$ queue where $\lambda_n=\lambda, \mu_n = n\mu$, $n\geq 0$.} The process has a stationary Poisson distribution 
\[
p_j = e^{-\lambda/\mu}(\lambda/\mu)^j/j!,\>j=0,1,\ldots
\]
The orthogonal polynomials are the Poisson-Charlier polynomials
\[
Q_n(z) = C_n(z/\mu;\lambda/\mu),\>n\geq 0 ,
\]
where $\{C_n(z;\nu)\}_{n=0}^\infty$ has a generating function
\[
\sum_{n=0}^\infty C_n(z;\nu)\frac{w^n}{n!} = e^w(1-w/\nu)^z.
\]
\item
\emph{The linear birth and death process where $\lambda_n=(n+\beta)\lambda$,
 $\mu_n=n\mu$, with $\lambda,\mu,\beta >0$.}  
The process arises from individuals which split at rate $\lambda$, die at rate $\mu$ and immigration of individuals occurs at rate $\lambda\beta$. Then
\[
\pi_j = \frac{\beta_{(j)}}{j!}      \Big (\frac{\lambda}{\mu}\Big )^j,\>j=0,1,\ldots
\]
There are three cases to consider.
\begin{enumerate}
\item
$\lambda < \mu$. The spectral polynomials are related to the Meixner polynomials by
\[
Q_n(z) = M_n\Big (\frac{z}{\mu-\lambda};\beta,\frac{\lambda}{\mu}\Big ),\>n=0,1,\ldots
\]
The polynomials are orthogonal on
\[
\Big (1-\frac{\lambda}{\mu}\Big )^\beta\frac{\beta_{(z)}}{z!}\Big (\frac{\lambda}{\mu}\Big )^z,\>z=0,1,\ldots
\]
 at points $(\mu-\lambda)z$, $z=0,1,\ldots$. The first point of increase is zero corresponding to 
$e^{0t}=1$ in the spectrum.
There is a negative binomial stationary distribution for the process
\[
p_i = \Big (1 - \frac{\lambda}{\mu}\Big )^{\beta}\frac{\beta_{(i)}}{i!}\Big (\frac{\lambda}{\mu}\Big )^i,\>i=0,1,\ldots
\]
The Meixner polynomials have a generating function
\begin{equation}
\sum_{n=0}^\infty M_n(x;a;q)\frac{a_{(n)}}{n!}z^n=(1-q^{-1}z)^x(1-z)^{-x-a}.
\label{Mgf:0}
\end{equation}
\item
$\lambda > \mu$. 
\[
Q_n(z) = \Big (\frac{\lambda}{\mu}\Big )^nM_n(\frac{z}{\lambda-\mu}-\beta;\beta,\frac{\mu}{\lambda}\Big ),\>n=0,1,\ldots
\]
The polynomials are orthogonal on
\[
\Big (1-\frac{\mu}{\lambda}\Big )^\beta\frac{\beta_{(z)}}{z!}\Big (\frac{\mu}{\lambda}\Big )^z,\>z=0,1,\ldots
\]
 at points $(z+\beta)(\lambda-\mu)$, $z=0,1,\ldots$. The first point of increase is $\beta (\lambda-\mu)$, corresponding to a spectral term $e^{-\beta (\lambda-\mu)t}$.
There is not a stationary distribution for the process in this case, with $\sum_{j=0}^\infty \pi_j=\infty$. 
\item
$\lambda=\mu$. The spectral polynomials are related to the Laguerre polynomials by
\[
Q_n(z) = \frac{n!}{\beta_{(n)}}L_n^{(\beta-1)}(z/\lambda),\>n\geq 0.
\]
In this case there is a continuous spectrum and the polynomials are orthogonal on the gamma distribution
\[
\frac{1}{\lambda^\beta\Gamma(\beta)}z^{\beta-1}e^{-z/\beta},\>z >0.
\]
There is no stationary distribution of the process in this case.
The Laguerre polynomials have a generating function 
\[
\sum_{n=0}^\infty L_n^{(\beta -1)}(x)z^n = (1-z)^{-\beta}\exp \{xz/(1-z)\}.
\]
\end{enumerate}
\item
\emph{A two urn model with $\lambda_n=(N-n)(a-n)$, $\mu_n=n\big (b-(N-n)\big )$, $n=0,1,\ldots ,N$, $a,b\geq N$.} The process arises from a model with two urns with $a$ and $b$ balls, with $N$ tagged balls. At an event two balls are chosen at random from the urns and interchanged. The state of the process is the number of tagged balls in the first urn. 
The spectral polynomials are related to the dual Hahn polynomials by
\[
Q_n(z) = R_n(\lambda(z);a,b,N),\>n=0,1,\ldots
\]
where
\[
 R_n(\lambda(z);a,b,N) = \phantom{x}_3F_2(-n,-z,z-a-b-1;-a,-N;1),
\]
orthogonal on 
\[
\frac{
{N-b-1\choose N}N!N_{[z]}a_{[z]}(2z-a-b-1)
}
{
z!b_{[z]}(z-a-b-1)_{(N+1)}
}
\]
with $\lambda(z) = z(z-a-b-1)$. There is a hypergeometric stationary distribution in the process of
\[
p_i = \frac{
{a\choose i}{b\choose N-i}
}
{
{a+b\choose N}
},
\>i=0,1,\ldots ,N.
\] 
\item
\emph{An Ehrenfest urn where $\lambda_n=(N-n)p$, $\mu_n=nq$, $0 \leq n \leq N$, $0 < p < 1$ and $q=1-p$.} 
The spectral polynomials are the Krawtchouk polynomials 
\[
Q_n(z) = K_n(z;N,p),\>0 \leq n \leq N,
\]
orthogonal on the Binomial $(N,p)$ distribution
\[
{N\choose z}p^zq^{N-z},\>z=0,1,\ldots N
\]
which is also the stationary distribution in the process.
\end{itemize}

\subsection{Composition birth and death processes}
Consider $N$ identically distributed birth and death processes $\{\xi_i(t)\}_{t \geq 0}$, $i=1,\ldots N$ each with state space $0,1,\ldots$. It is assumed that there is no absorbing state at $-1$ and $\lambda_0>0$. The transition functions for the labelled processes are
$p_{\bm{i}\bm{j}}(t) := \prod_{k=1}^Np_{i_k,j_k}(t)$.
In composition Markov processes interest is in the unlabelled configuration of $\bm{\xi}(t)$ specified by $\bm{X}(t)$, where
\[
X_k(t) = |\{i_j = k, j=1,\ldots ,N \}|
\]
for $k=0,1\ldots $.
The probability generating function of $\bm{X}(t)$ conditional on $\bm{X}(0)=\bm{x}$ is
\begin{equation}
\mathbb{E}\big [ \prod_{k=1}^ds_k^{X_k(t)}\big ] =
\prod_{i=0}^d\Bigg (\sum_{j=0}^dp_{ij}(t)s_j\Bigg )^{x_i}
\label{pgfxt:0}
\end{equation}
where possibly there are a countable infinity of states with $d=\infty$.
Transitions and rates are, for $j=0,1,\ldots$,
\begin{equation}
\bm{x} \to 
\begin{cases}
\bm{x} + \bm{e}_{j+1} - \bm{e}_j &\mbox{rate~} x_j\lambda_j,\\
\bm{x} + \bm{e}_{j-1} - \bm{e}_{j} &\mbox{rate~} x_j\mu_j.
\label{rates:2}
\end{cases}
\end{equation}
The total rate from $\bm{x}$ is $\sum_{j\geq 0} x_j(\lambda_j+\mu_j)$.
$\{\bm{X}(t)\}_{t\geq 0}$ is reversible with respect to $\widetilde{m}(\bm{x};\bm{\pi}) = {\bm{N}\choose \bm{x}}\prod_{j=1}^d\pi_j^{x_j}$  in the sense that
\begin{eqnarray*}
\widetilde{m}(\bm{x};\bm{\pi})\lambda_jx_j &=& \widetilde{m}(\bm{x}+\bm{e}_j;\bm{\pi})\mu_{j+1}x_{j+1},\>j=0,1,\ldots\\
\widetilde{m}(\bm{x};\bm{\pi})\mu_jx_j &=& \widetilde{m}(\bm{x}-\bm{e}_j;\bm{\pi})\lambda_{j-1}x_{j-1},\>j=1,2\ldots .
\end{eqnarray*}

\begin{Theorem}
If the spectrum is discrete, with support $\{\zeta_l\}_{l\geq 0}$, $\mu_0=0$, $\zeta_0=0$, and a stationary distribution exists, then
\begin{equation}
p(\bm{x},\bm{y};t) = m(\bm{y},\bm{p})\Bigg \{1 + \sum_{\{\bm{n}:0 < |\bm{n}| \leq N\}}e^{-t \sum_{i\geq 1}n_i\zeta_i}{N\choose \bm{n}}^{-1}Q_{\bm{n}}(\bm{x};\bm{u})Q_{\bm{n}}(\bm{y};\bm{u})\Bigg \},
\label{transitionfns:5}
\end{equation}
where $\{Q_{\bm{n}}(\bm{x};\bm{u})\}$ are the multivariate Krawtchouk polynomials with 
\begin{equation}
u^{(l)}_i = Q_i(\zeta_l)\sqrt{\psi(\zeta_l)/\psi(0)},\>i,l=0,1,\ldots.
\label{uQ:0}
\end{equation}
The indexing in elements of $\bm{x},\bm{y}$ now begins at $0$.
If the spectrum is discrete, with support   $\{\zeta_l\}_{l\geq 0}$, $\mu_0=0$  then
\begin{equation}
p(\bm{x},\bm{y};t) = \widetilde{m}(\bm{y};\bm{\pi})\sum_{\{\bm{n}:0 \leq |\bm{n}| \leq N\}}e^{-t \sum_{i\geq 0}n_i\zeta_i}{N\choose \bm{n}}^{-1}Q_{\bm{n}}(\bm{x};\bm{u})Q_{\bm{n}}(\bm{y};\bm{u}),
\label{transitionfns:6}
\end{equation}
where 
 $\{Q_{\bm{n}}(\bm{x};\bm{u})\}$ are the multivariate Krawtchouk polynomials with 
\begin{equation}
u^{(l)}_i = Q_i(\zeta_l)\sqrt{\psi(\zeta_l)},\>i,l=0,1,\ldots.
\label{uQ:1}
\end{equation}
In this case $\zeta_0 > 0$,  $u^{(0)}$ is not identically 1, and 
\[
\sum_{i\geq 0}u^{(k)}_iu^{(l)}_i\pi_i = \delta_{kl},\>k,l=0,1,\ldots.
\]
This covers the case when a stationary distribution does exist and also when a stationary distribution does not exist because $\sum_{k=0}^\infty\pi_k = \infty$.
\end{Theorem}
\begin{proof}
The probabilistic structure of $\{\bm{X}(t)\}_{t\geq 0}$ with probability generating function (\ref{pgfxt:0}) implies that the multivariate Krawtchouk polynomials are the eigenfunctions of the transition distribution. Indexing in $\bm{X}(t)$ is from 0, rather than the usual indexing from 1. From the Karlin and McGregor spectral expansion (\ref{KMG:0})
\begin{eqnarray}
p_{ij}(t) &=& \pi_j\Big \{ \sum_{k\geq 0}e^{-t\zeta_k}Q_i(\zeta_k)Q_j(\zeta_k)\psi(\{\zeta_k\})\Big \}
\nonumber \\
p_{ij}(t) &=& \psi(\{0\})\pi_j\Big \{1 + \sum_{k\geq 1}e^{-t\zeta_k}Q_i(\zeta_k)Q_j(\zeta_k)\psi(\{\zeta_k\})/\psi(\{0\})\Big \}
\nonumber \\
&=& p_j\Big \{1 + \sum_{k\geq 1}e^{-t\zeta_k}u_i^{(k)}u_j^{(k)}\Big \},
\label{cmc:0}
\end{eqnarray}
where $\{u_i^{(k)}\}$ is defined in (\ref{uQ:0}) and satisfies
\[
\sum_{i\geq 0}u_i^{(k)}u_i^{(l)}p_i = \delta_{kl},\>\>k,l\geq 0.
\]
The second case (\ref{transitionfns:6}) follows similarly.  The multivariate Krawtchouk polynomials then have a generating function
\begin{equation}
G(\bm{x},\bm{w}, \bm{u}) 
=
 \prod_{j\geq 0}\Big (u^{(0)}_j + \sum_{l\geq 1}^{d-1}w_lu_j^{(l)}\Big )^{x_j}.
\label{main_gf:1}
\end{equation}
\end{proof}
The transition probability expansion (\ref{transitionfns:6}) can be written in a Karlin and McGregor spectral expansion form where the dual polynomials are important.
 Denote $\widetilde{u}_i^{(l)}=Q_i(\xi_l)$, $i,l=0,1,\ldots$;
${\cal Q}_{\bm{x}}(\bm{\nu};\widetilde{u})={N\choose \bm{\nu}}^{-1}Q_{\bm{\nu}}(\bm{x};\widetilde{\bm{u}})$; and a multinomial spectral measure (which is a probability measure)
\begin{equation}
\widetilde{m}(\bm{\nu};\psi) = {N\choose \bm{\nu}}\psi(\zeta_0)^{\nu_0}\psi(\zeta_1)^{\nu_1}\cdots,\>\nu_0+\nu_1+\cdots = N
\label{phispect}
\end{equation}
Then (\ref{transitionfns:6}) can be expressed as a spectral expansion
\begin{equation}
p(\bm{x},\bm{y}:t) 
= \widetilde{m}(\bm{y};\bm{\pi})
\sum_{\{\bm{\nu};0 \leq |\bm{\nu}| \leq N\}}e^{-t \sum_{i\geq 0}\nu_i\zeta_i}
{\cal Q}_{\bm{x}}(\bm{\nu};\widetilde{\bm{u}}){\cal Q}_{\bm{y}}(\bm{\nu};\widetilde{\bm{u}})
\widetilde{m}(\bm{\nu};\psi).
\label{transitionfns:10}
\end{equation}
The generating function of the dual polynomials 
\begin{eqnarray}
H(\bm{n},\bm{v},\widetilde{\bm{u}}) &=& \sum_{\{\bm{x}:|\bm{x}|=N\}}
{N\phantom{^+}\choose \bm{n}^+}^{-1}{N\choose \bm{x}}Q_{\bm{n}}(\bm{x},\widetilde{\bm{u}})\prod_{i\geq 0}v_i^{x_i}
\nonumber \\
&=& \prod_{l\geq 0}\Bigg (v_0 + \sum_{j\geq 1}Q_j(\zeta_l)v_j\Bigg )^{n_l}
\nonumber \\
&=& \prod_{k=1}^N\Bigg (v_0 + \sum_{j\geq 1}Q_j(Z_k)v_j\Bigg )
\label{dual:20}
\end{eqnarray}
where in this generating function $\bm{n}(\bm{Z})$ is regarded as a random variable by taking
\begin{equation}
n_l(\bm{Z}) = |\{Z_k:Z_k = \zeta_l, k=1,\ldots ,N\}|.
\label{nz:0}
\end{equation}
$\{Z_k\}_{k=1}^N$ are independent identically distributed random variables with probability measure $\psi$.  Without loss of take generality $v_0=1$ in (\ref{dual:20}) and consider coefficients of $\prod_{i\geq 1}v_i^{x_i}$, indexing the dual polynomial by $(x_1,x_2,\ldots)$ with $x_1+x_2+\cdots \leq N$. Note the scaling that the dual polynomials is 1 when $x_i=0$, $i \geq 1$.
\begin{Theorem}
Define
\[
{\cal N}_j = \sum_{k=1}^NQ_j(Z_k) = \sum_{l\geq 0}n_lQ_j(\zeta_l),\>j\geq 1.
\]
${N\phantom{^+}\choose \bm{n}^+}^{-1}Q_{\bm{n}}(\bm{x},\widetilde{\bm{u}})$
 is a polynomial of degree $x_1+x_2+\cdots$ in $\{{\cal N}_j\}_{j\geq 1}$ 
whose only term of maximal degree  is $\prod_{j\geq 1}{\cal N}_j^{x_j}$.
The total degree of $\bm{Z}$ in the dual polynomials indexed by $(x_1,x_2,\ldots)$ is $\sum_{j\geq 1}jx_j$ with a single leading term of this degree.
\end{Theorem}
\begin{proof}
The proof of the first statement follows from Theorem \ref{thm:dual}. The proof of the second statement is immediate by knowing that ${\cal N}_j$ is of degree $j$ in $\bm{Z}$.
\end{proof}
The third case of linear birth and death processes composition Markov chains is interesting as having a continuous spectral measure which is a product measure of $N$ gamma distribution measures. The spectral polynomials are well defined by a generating function as coefficients
 of $\prod_{j=1}^\infty v_j^{x_j}$ in
\[
\prod_{k=1}^N \Bigg ( 1 + \sum_{j\geq 1}Q_j(Z_k)v_j\Bigg ),
\]
however elements of $\{Z_k\}_{k=1}^N$ are distinct, being continuous random variables, and the dual of the dual system is products of dual Laguerre polynomials which are not grouped to an index $\bm{n}$ as when there is a discrete spectrum.
%%%

The polynomials in the Meixner class (\ref{mclass:0}) are additive in the sense that if $\{Q_m^N(|\bm{z}|)\}$ are the orthogonal polynomials on the distribution of $|\bm{Z}|$ then
the generating function for these polynomials is 
\begin{equation}
G^N(v,|\bm{z}|) = h(v)^Ne^{|\bm{z}|u(v)}=\sum_{m=0}^\infty Q^N_m(|\bm{z}|)v^m/m!
\label{mclass:1}
\end{equation}
and
\begin{equation}
Q^N_m(|\bm{z}|) = \sum_{\{\bm{m}:|\bm{m}|=m\}}{m\choose \bm{m}}\prod_{j=1}^NQ_{m_j}(z_j).
\label{mclass:4}
\end{equation}
This additivity implies an interesting identity. 
\begin{Theorem}
The dual multivariate Krawtchouk polynomials with generating function (\ref{dual:20}) satisfy the identity
\begin{equation}
{N\choose \bm{n}}^{-1}\sum_{\{\bm{x}:\sum_{j=1}^\infty jx_j = m\}}
{N\choose \bm{x}}\frac{m!}{\prod_{j=1}^\infty {j!}^{x_j} }Q_{\bm{n}}(\bm{x},\widetilde{\bm{n}}) = Q^N_m(|\bm{Z}|),
\label{mclass:5}
\end{equation}
where $\bm{x}= (x_1,x_2,\ldots)$.
In this equation $\bm{n}=\bm{n}(\bm{Z})$ is regarded as a random variable in the sense of (\ref{nz:0}).
\end{Theorem}
\begin{proof}
Set $v_j=v^j/j!$, $j=0,1,\ldots $ in (\ref{dual:20}). Then 
\begin{eqnarray*}
\sum_{\{\bm{x}:|\bm{x}|=N\}}
{N\phantom{^+}\choose \bm{n}^+}^{-1}{N\choose \bm{x}}Q_{\bm{n}}(\bm{x},\widetilde{\bm{u}})\frac{v^{\sum_{j=1}^\infty  jx_j}}{\prod_{j=1}^\infty{j!}^{x_j}}
%\nonumber \\
&=&\prod_{k=1}^N\Bigg (\sum_{j\geq 0}Q_j(Z_k)v^j/j!\Bigg )
\nonumber \\
&=&h(v)^Ne^{|\bm{Z}|u(v)}
\nonumber \\
&=&\sum_{m=0}^\infty Q_m^N(|\bm{Z}|)v^m/m!
\end{eqnarray*}
The theorem then follows by equating coefficients of $v^m$ on both sides of the generating function.
\end{proof}

%%%%%%%%%%%%%%%%%%%
\end{document}